\documentclass[pdflatex,sn-mathphys-num]{sn-jnl} %,sn-mathphys-num

\usepackage{graphicx}%
\usepackage{multirow}%
\usepackage{amsmath,amssymb,amsfonts}%
\usepackage{amsthm}%
\usepackage{mathrsfs}%
\usepackage[title]{appendix}%
\usepackage{xcolor}%
\usepackage{textcomp}%
\usepackage{manyfoot}%
\usepackage{booktabs}%
\usepackage{algorithm}%
\usepackage{algorithmicx}%
\usepackage{algpseudocode}%
\usepackage{listings}%

\theoremstyle{thmstyleone}%
\newtheorem{theorem}{Theorem}[section]
\newtheorem*{theorem*}{Theorem}
\newtheorem{corollary}[theorem]{Corollary}

\newtheorem{proposition}[theorem]{Proposition}
\newtheorem*{proposition*}{Proposition}

\theoremstyle{thmstyletwo}%

\newtheorem{example}[theorem]{Example}

\newtheorem{remark}[theorem]{Remark}

\theoremstyle{thmstylethree}%
\newtheorem{definition}[theorem]{Definition}

\numberwithin{equation}{section}
\DeclareFontFamily{U}{rsfs}{}
\DeclareFontShape{U}{rsfs}{m}{n}{<-6> rsfs5 <6-8> rsfs7 <8-> rsfs10}{}
\DeclareMathAlphabet{\mathscr}{U}{rsfs}{m}{n}

\newcommand{\rg}{{\rm g}}

\def\cF{\mathcal{F}}

\DeclareMathOperator\Div{div}

\DeclareMathOperator{\Stab}{Stab}

\newcommand{\Diff}{\mathrm{Diff}}
\newcommand{\End}{{\mathrm{End}}}

\renewcommand{\epsilon}{\varepsilon}

\newcommand{\Isom}{\mathrm{Isom}}

\newcommand{\Ric}{{\rm Ric}}

\def\g2{\varphi}
\def\s7{\Phi}

\def\fso{\mathfrak{so}}

\newcommand{\SO}{{\rm SO}}

\newcommand{\SU}{{\rm SU}}
\newcommand{\GL}{\mathrm{GL}}

\newcommand{\SL}{\mathrm{SL}}
\newcommand{\U}{{\rm U}}

\def\G2{\mathrm{G}_2}
\def\S7{\mathrm{Spin}(7)}
\def\Spin7{\mathrm{Spin(7)}}

\newcommand{\bR}{\mathbb{R}}

\newcommand{\fh}{\mathfrak{h}}
\newcommand{\fm}{\mathfrak{m}}

\newcommand{\Id}{\mathrm{Id}}

\begin{document}

\title[Topology of isometric classes and flows of geometric structures]{Topology of isometric classes and flows of geometric structures}

\author[1]{\fnm{Daniel} \sur{Fadel}}\email{daniel.fadel@icmc.usp.br}
\affil[1]{%
\orgdiv{Instituto de Ci\^encias Matem\'aticas e de Computa\c{c}\~ao (ICMC)}, 
\orgname{Universidade de S\~ao Paulo (USP)}, 
\orgaddress{%
\street{Avenida Trabalhador S\~ao-carlense, 400}, 
\city{S\~ao Carlos}, 
\postcode{13566-590}, 
\state{SP}, 
\country{Brazil}}}

\author*[2]{\fnm{Eric} \sur{Loubeau}}\email{loubeau@univ-brest.fr}
\affil*[2]{\orgdiv{D{\'e}partement de Math{\'e}matiques / LMBA UMR 6205}, \orgname{Universit{\'e} de Bretagne Occidentale}, \orgaddress{\street{6, avenue Victor Le Gorgeu}, \city{Brest}, \postcode{29200}, \country{France}}}

\equalcont{A Sorin, per l'amicizia, per le discussioni, per i pranzi.}

\abstract{
We study the topology of isometric classes of tensorial $H$-structures for closed and connected subgroups $H\leqslant\SO(n)$. The map assigning an $H$-structure its induced metric admits a global trivialization; consequently, the space of $H$-structures is homotopy equivalent to any fixed isometric class. On flat tori, these classes are mapping spaces into $\SO(n)/H$. We give examples with infinitely many connected components for $\mathrm{G}_2$, $\mathrm{Spin}(7)$, and almost Hermitian structures, determine the number of components for $\SU(m)$-structures on flat tori, and study the corresponding moduli spaces.

For closed manifolds of dimension $n>2$, the intrinsic torsion energy on the full space of $H$-structures has infimum zero on every nonempty path component, and its critical points are precisely the torsion-free structures. In fixed isometric classes, we contrast the concentration phenomena underlying finite-time singularities with cohomological classes admitting positive energy lower bounds and harmonic representatives.

Finally, we establish a lifting principle for metric-dependent flows, derive an evolution identity for isometric flows, and extend harmonic-flow estimates to general $H$-structures.
}



\maketitle

\section{Introduction}\label{sec: intro}

Many geometric structures arising in differential geometry may be described as reductions of the oriented frame bundle of a manifold to a closed and connected Lie subgroup $H\leqslant\SO(n)$. Examples include almost Hermitian, $\SU(m)$, $\mathrm{G}_2$, and $\mathrm{Spin}(7)$ structures. In \cite{Fadel2022}, a general framework was developed for the study of flows of tensorial $H$-structures, unifying several phenomena previously known in particular geometries. After establishing a general theory of deformations of $H$-structures and their interaction with the induced Riemannian metric and intrinsic torsion, the theory was specialised to the harmonic flow of $H$-structures, namely the negative gradient flow of the intrinsic torsion energy inside a fixed isometric class.

The harmonic flow preserves the induced Riemannian metric and therefore evolves within a fixed isometric class. More generally, isometric flows of $H$-structures are naturally defined on the space of compatible structures inducing a given metric. This motivates the study of the topology of such spaces.

A central observation of this paper is that, once a manifold admits an $H$-structure, compatibility with a Riemannian metric imposes essentially no additional topological restriction. More precisely, if
\begin{equation*}
\Pi:\Gamma(\cF)\longrightarrow \mathrm{Met}(M),
\qquad
\Pi(\xi)=g_\xi,
\end{equation*}
denotes the map assigning to an $H$-structure its induced Riemannian metric, then Theorem~\ref{thm: homotopy_general} gives a global trivialization of $\Pi$. Since the space $\mathrm{Met}(M)$ of Riemannian metrics is contractible, Corollary~\ref{cor:metric_map_homotopy_equivalence} shows that the full space of $H$-structures is homotopy equivalent to any fixed isometric class. Consequently, from the homotopical point of view, the topology of the space of geometric structures is entirely encoded by the corresponding isometric classes.

For parallelizable manifolds, these isometric classes are identified with mapping spaces into $\SO(n)/H$ after choosing a global orthonormal frame. In particular, for every Riemannian metric $g$ on $T^n$ one obtains
\begin{equation*}
\pi_0(\langle g\rangle_H)
\cong
[T^n,\SO(n)/H].
\end{equation*}
We consider almost Hermitian, $\SU(m)$, $\G2$, and $\mathrm{Spin}(7)$ structures on flat tori. The space of compatible $\SU(3)$-structures on $T^6$ has exactly $64$ connected components for any metric, whereas the corresponding spaces for $\SU(m)$ with $m=2$ or $m\geq4$ have infinitely many (Example~\ref{ex:SU_m_tori}). Section~\ref{subsec: moduli} studies quotients by isometries and diffeomorphisms.

For closed manifolds, we then relate this topology to the variational theory of harmonic $H$-structures. First, Theorem~\ref{thm:unrestricted_torsion_energy_degeneracy} shows that the intrinsic torsion energy on the unrestricted space of $H$-structures is scale-degenerate in dimensions $n>2$: its infimum is zero on every nonempty path component, and its only critical points are torsion-free structures. Thus the fixed-isometric-class restriction is not merely natural from the point of view of the harmonic flow; it is also necessary in order to remove the artificial homothetic collapse direction. Within a fixed isometric class, topology becomes genuinely variational. In Section~\ref{subsec: finite-time_sing}, we reinterpret the finite-time singularity examples from \cite{Fadel2022} as concentration phenomena in nontrivial isometric homotopy classes with vanishing energy infimum. By contrast, for $\mathrm U(3)$-structures on a flat $6$-torus, isometric homotopy classes detected by nonzero elements of $H^2(T^6;\mathbb Z)$ have positive energy lower bounds. Proposition~\ref{prop:holomorphic_U3_harmonic_structures} and Corollary~\ref{cor:abelian_threefold_harmonic_U3} construct smooth harmonic representatives in some of these classes on abelian threefolds, using holomorphic maps into $\mathbb{CP}^3$.

The final section concerns metric-dependent flows and the analytic theory of the harmonic $H$-flow. Proposition~\ref{prop:general_lifting_metric_flows} shows that a smooth solution of a metric flow lifts to a flow of $H$-structures whenever the latter is determined by a symmetric endomorphism depending only on the induced metric. The Ricci $H$-flow of \cite[Example~1.26 and Lemma~1.27]{Fadel2022} is a particular case. The lift is constructed using the evolution operator of a time-dependent linear ODE on bundle automorphisms.

We also establish the evolution identity \eqref{eq:postscriptum-258-general} for arbitrary closed and connected subgroups $H\leqslant\SO(n)$ and apply it to the harmonic flow. The analytic estimates of \cite{Fadel2022} are then discussed without the additional structural condition \cite[Condition~(2.1)]{Fadel2022}, using a two-sided comparison between $|\nabla\xi|^2$ and $|T|^2$ and the intrinsic torsion energy as the monotone functional.

\section{Preliminaries}

Throughout this paper, $M^n$ will denote a connected and oriented smooth $n$-manifold without boundary, with $n>2$. We let $\mathrm{Fr}^{+}(M)$ denote the frame bundle of oriented frames on $M$, that is, the principal $\mathrm{GL}^{+}(n,\mathbb{R})$-bundle whose fibre over $x\in M$ consists of the oriented linear isomorphisms $u:T_xM\to\mathbb{R}^n$, with right action $\mathrm{GL}^{+}(n,\mathbb{R})\times \mathrm{Fr}^{+}(M)\to \mathrm{Fr}^{+}(M)$ given by $(\rg,u)\mapsto \rg.u:=\rg^{-1}\circ u$. We shall use the notations and conventions of \cite{Fadel2022}. Unless otherwise stated, all spaces of smooth sections are endowed with the standard $C^\infty$ topology.

Whenever we speak about an \textbf{$H$-structure on $M^n$}, it will always mean an $H$-structure compatible with the fixed orientation on $M^n$, that is, a principal $H$-subbundle $Q\subset\mathrm{Fr}^{+}(M)$, where $H\leqslant\SO(n)$ is a closed and connected Lie subgroup. A choice of Riemannian metric $g$ on $M$ determines the $\SO(n)$-structure $\mathrm{Fr}(M,g)\subset\mathrm{Fr}^{+}(M)$ consisting of all oriented $g$-orthonormal frames. Conversely, an $H$-structure $Q\subset\mathrm{Fr}^{+}(M)$ induces the $\SO(n)$-structure $P:=\SO(n)\cdot Q$ containing $Q$, and hence a unique Riemannian metric $g$ on $M$ such that $P=\mathrm{Fr}(M,g)$.

In general, many distinct $H$-structures may induce the same Riemannian metric. If $g$ is fixed, then an $H$-structure $Q$ induces $g$ if and only if $Q$ is an $H$-reduction of $\mathrm{Fr}(M,g)$; such a structure will be called an \textbf{$H$-structure compatible with $g$}. The quotient $\mathrm{Fr}(M,g)/H$ is a fibre bundle over $M$ with fibre $\SO(n)/H$, and compatible $H$-structures are naturally identified with sections of this homogeneous bundle. Indeed, given a compatible $H$-structure $Q\subset\mathrm{Fr}(M,g)$, one defines $\sigma_Q(x):=\pi_H(u)$ for any $u\in Q_x$, where $\pi_H:\mathrm{Fr}(M,g)\to\mathrm{Fr}(M,g)/H$ is the quotient map. Conversely, any section $\sigma\in\Gamma(\mathrm{Fr}(M,g)/H)$ determines the compatible $H$-structure
\begin{equation*}
    Q_\sigma:=\pi_H^{-1}(\sigma(M)).
\end{equation*}
More generally, arbitrary $H$-structures on $M$ correspond to sections of the $\mathrm{GL}^{+}(n,\mathbb{R})/H$-bundle $\mathrm{Fr}^{+}(M)/H\to M$.

Throughout the paper we consider closed and connected subgroups $H\leqslant\SO(n)$ realized as the stabiliser
\begin{equation*}
H=\Stab(\xi_\circ)=\{A\in\GL^+(n,\mathbb R):A\cdot\xi_\circ=\xi_\circ\}
\end{equation*}
of a tensor (or tuple of tensors) $\xi_\circ$ in a finite-dimensional tensor representation $V$ of $\GL^+(n,\mathbb R)$, as in \cite[Section~1.1]{Fadel2022}. Let $\mathcal O_{\xi_\circ}=\GL^+(n,\mathbb R)\cdot\xi_\circ\cong\GL^+(n,\mathbb R)/H$. The bundle
\begin{equation}\label{eq:orbit_bundle}
\cF:=\mathrm{Fr}^{+}(M)\times_{\GL^+(n,\mathbb R)}\mathcal O_{\xi_\circ}
\cong\mathrm{Fr}^{+}(M)/H
\end{equation}
is a subbundle of the associated tensor vector bundle with fibre $V$. Its sections are precisely the tensor fields pointwise modelled on $\xi_\circ$, and hence the $H$-structures on $M$.

Henceforth, we shall use the tensor notation $\xi$ for an $H$-structure $Q$, and denote by $Q_\xi\subset\mathrm{Fr}^{+}(M)$ the corresponding $H$-reduction and by $g_\xi$ the unique Riemannian metric induced by $\xi$. Accordingly, we shall call $\Gamma(\cF)$ the \textbf{space of $H$-structures} on $M$, which may also be identified with $\Gamma(\mathrm{Fr}^{+}(M)/H)$. We also let
\begin{equation*}
\Pi:\Gamma(\mathcal F)\longrightarrow \mathrm{Met}(M),
\qquad
\Pi(\xi):=g_\xi,
\end{equation*}
denote the map assigning to each $H$-structure its induced Riemannian metric.

\begin{definition}
    Two $H$-structures $\xi_1$ and $\xi_2$ on $M^n$ are said to be \textbf{isometric} if they induce the same Riemannian metric, that is, if $g_{\xi_1}=g_{\xi_2}$. Given an $H$-structure $\xi$ on $M^n$, the space of all $H$-structures on $M^n$ inducing the same metric as $\xi$ will be called the \textbf{isometric class} of $\xi$ and denoted by $[[\xi]]$. If $g=g_\xi$ is the induced metric, we shall also denote this same space by $\langle g\rangle_H$.
\end{definition}

By the previous discussion, if $\xi$ induces the metric $g$, then elements of $[[\xi]]=\langle g\rangle_H$ are naturally identified with sections of the homogeneous bundle $\mathrm{Fr}(M,g)/H\to M$. Thus, from the point of view of geometric flows, the spaces $\langle g\rangle_H$ should be regarded as the natural configuration spaces for isometric deformations of geometric structures.

\section{The topology of isometric classes of $H$-structures}
\label{sec: topology_of_H_structures}

In this section, we study the topology of the space of $H$-structures through the topology of its isometric classes. The guiding point is that many geometric flows of $H$-structures, such as the harmonic flow considered in \cite{Fadel2022}, evolve inside a fixed metric class. Thus the spaces $\langle g\rangle_H$ should be regarded as the natural configuration spaces for isometric deformations of geometric structures.

The main result of this section is that this restriction to a fixed metric class loses no homotopical information: once $M$ admits an $H$-structure, the inclusion of any isometric class into the full space of $H$-structures is a homotopy equivalence. We then make this statement concrete for parallelizable manifolds and, in particular, for flat tori, where the problem reduces to the topology of mapping spaces into the homogeneous spaces $\SO(n)/H$.

\subsection{The metric map and homotopy lifting}

\begin{definition}\label{def: homotopy}
Let $\xi$ be an $H$-structure inducing a metric $g$. The \textbf{homotopy class} of $\xi$ is its path-connected component in $\Gamma(\cF)$. Its \textbf{isometric homotopy class} $[\xi]$ is its path-connected component in $\langle g\rangle_H=\Pi^{-1}(g)$. Thus an isometric homotopy is a path of $H$-structures inducing the same metric at every time.
\end{definition}

Our first result identifies the metric map globally. In particular, varying the metric does not impose any new topological restriction on an existing $H$-structure.

\begin{theorem}[Trivialization of the metric map]\label{thm: homotopy_general}
Suppose that $M$ admits an $H$-structure. For every reference metric $g_0\in\mathrm{Met}(M)$, the map
\begin{equation*}
\Pi:\Gamma(\cF)\longrightarrow\mathrm{Met}(M),\qquad \Pi(\xi)=g_\xi,
\end{equation*}
admits a global trivialization
\begin{equation}\label{eq:metric_map_trivialization}
\Gamma(\cF)\cong \mathrm{Met}(M)\times\langle g_0\rangle_H,
\end{equation}
under which $\Pi$ corresponds to the first projection. Consequently, $\Pi$ is surjective and has the smooth parametric path-lifting property.
\end{theorem}

\begin{proof}
Fix $g_0\in\mathrm{Met}(M)$. For each metric $g$, there is a unique $g_0$-self-adjoint positive-definite endomorphism $A_g\in\Gamma(\End(TM))$ such that
\begin{equation}\label{eq:relative_metric_endomorphism}
 g(X,Y)=g_0(A_gX,Y),\qquad X,Y\in\Gamma(TM).
\end{equation}
Define $B_g=A_g^{1/2}$, the positive $g_0$-self-adjoint square root. Since $B_g^2=A_g$,
\begin{equation}\label{eq:relative_metric_sqrt}
 g(X,Y)=g_0(B_gX,B_gY).
\end{equation}
The positive square-root map is smooth on the open cone of positive-definite self-adjoint endomorphisms. Therefore $B_g$ depends smoothly on both $g$ and the point of $M$. Moreover, $B_g$ is invertible and has positive determinant, so it is an orientation-preserving bundle automorphism of $TM$. Notice that $B_{g_0}=\Id$.

Recall that the natural action of a bundle automorphism $F$ on a tensorial $H$-structure satisfies
\begin{equation}\label{eq:metric_bundle_action}
 g_{F\cdot\xi}(X,Y)=g_\xi(FX,FY).
\end{equation}
This identity also follows directly from the corresponding action on the $H$-reduction of the oriented frame bundle. In particular, if $\eta\in\langle g_0\rangle_H$, then \eqref{eq:relative_metric_sqrt} and \eqref{eq:metric_bundle_action} imply that $B_g\cdot\eta$ induces $g$. Thus the map
\begin{equation}\label{eq:metric_product_map}
 \Psi:\mathrm{Met}(M)\times\langle g_0\rangle_H\longrightarrow\Gamma(\cF),
 \qquad \Psi(g,\eta)=B_g\cdot\eta,
\end{equation}
is well-defined and satisfies $\Pi\circ\Psi(g,\eta)=g$.

Conversely, let $\xi\in\Gamma(\cF)$ and put $g=g_\xi$. The structure $B_g^{-1}\cdot\xi$ induces $g_0$, since
\begin{align*}
 g_{B_g^{-1}\cdot\xi}(X,Y)
 &=g(B_g^{-1}X,B_g^{-1}Y)\\
 &=g_0(X,Y).
\end{align*}
It follows that
\begin{equation}\label{eq:metric_product_inverse}
 \Psi^{-1}(\xi)=\bigl(g_\xi,B_{g_\xi}^{-1}\cdot\xi\bigr).
\end{equation}
The formulas \eqref{eq:metric_product_map} and \eqref{eq:metric_product_inverse} define inverse maps. The assignment $\xi\mapsto g_\xi$, the positive square-root map, and the action of bundle automorphisms on tensor fields are smooth in the $C^\infty$ topologies. Hence $\Psi$ is a smooth trivialization, proving \eqref{eq:metric_map_trivialization}.

Finally, the first projection of a product has the parametric path-lifting property. Transporting its lifts through $\Psi$ proves the last assertion. No compactness assumption on $M$ is needed.
\end{proof}

The product description also gives an explicit formula for a lift with varying initial conditions. We record it in the same form used below.

\begin{proposition}[Parametric path lifting]\label{prop:parametric_path_lifting}
Let $P$ be a Fr\'echet manifold, let $(g_{p,t})_{(p,t)\in P\times[0,1]}$ be a smooth family of Riemannian metrics, and let $\xi_{0,p}$ be a smooth family of $H$-structures satisfying $g_{\xi_{0,p}}=g_{p,0}$. Then there exists a smooth family $\xi_{p,t}$ of $H$-structures such that
\begin{equation*}
 \xi_{p,0}=\xi_{0,p},\qquad g_{\xi_{p,t}}=g_{p,t}
 \quad\text{for all }(p,t)\in P\times[0,1].
\end{equation*}
\end{proposition}

\begin{proof}
Fix any reference metric $g_0$ and use the operators $B_g$ defined in the proof of Theorem~\ref{thm: homotopy_general}. Set
\begin{equation*}
 \eta_p:=B_{g_{p,0}}^{-1}\cdot\xi_{0,p}\in\langle g_0\rangle_H,
 \qquad \xi_{p,t}:=B_{g_{p,t}}\cdot\eta_p.
\end{equation*}
By \eqref{eq:metric_bundle_action}, $\xi_{p,t}$ induces $g_{p,t}$. At $t=0$ the two bundle automorphisms cancel, so $\xi_{p,0}=\xi_{0,p}$. Smoothness follows from the same square-root construction as before, now applied to the smooth family $(p,t,x)\mapsto g_{p,t}(x)$. In particular, the construction is pointwise on $M$ and imposes no compactness condition on the parameter manifold $P$.
\end{proof}

\begin{remark}\label{rem:two_metric_lifts}
The preceding lift depends only on the metrics at the endpoints of each parameter interval. There is also a useful evolution-based construction. Given a smooth path $g(t)$ with $g(0)=g_{\xi_0}$, define the $g(t)$-self-adjoint endomorphism
\begin{equation*}
 A(t):=\frac12 g(t)^{-1}\dot g(t),\qquad
 g(t)(A(t)X,Y)=\frac12\dot g(t)(X,Y).
\end{equation*}
Solve the non-autonomous linear ODE
\begin{equation}\label{eq:alternative_metric_path_ODE}
 \dot F(t)=A(t)\circ F(t),\qquad F(0)=\Id,
\end{equation}
and set $\xi(t)=F(t)\cdot\xi_0$. The induced metric $q(t)=g_{\xi(t)}$ satisfies
\begin{equation*}
 \dot q(t)(X,Y)=q(t)(A(t)X,Y)+q(t)(X,A(t)Y).
\end{equation*}
The prescribed $g(t)$ satisfies this same linear equation and the same initial condition, whence $q(t)=g(t)$ by uniqueness. The square-root construction and the ODE construction generally give different lifts. The former provides a trivialization of the metric map, while the latter is adapted to evolution equations, as in Proposition~\ref{prop:general_lifting_metric_flows}.

For the straight-line path $g(t)=(1-t)g_0+tg_1$, however, the two constructions coincide. Let $K$ be the positive $g_0$-self-adjoint endomorphism defined by $g_1(X,Y)=g_0(KX,Y)$, and put $B_t=(1-t)\Id+tK$. Then
\begin{equation*}
 g(t)(X,Y)=g_0(B_tX,Y),\qquad
 A(t)=\frac12 B_t^{-1}(K-\Id).
\end{equation*}
Since $B_t$ commutes with $K$, differentiation of its positive square root gives
\begin{equation*}
 \frac{d}{dt}B_t^{1/2}=\frac12 B_t^{-1}(K-\Id)B_t^{1/2}
 =A(t)B_t^{1/2}.
\end{equation*}
As $B_0=\Id$, uniqueness in \eqref{eq:alternative_metric_path_ODE} yields
$F(t)=B_t^{1/2}$, which is the square-root lift of the prescribed path.
\end{remark}

\subsection{Homotopy type of isometric classes}

\begin{corollary}\label{cor:metric_map_homotopy_equivalence}
Suppose $M$ admits an $H$-structure. For every $g_0\in\mathrm{Met}(M)$, the inclusion
\begin{equation*}
 \langle g_0\rangle_H\hookrightarrow\Gamma(\cF)
\end{equation*}
is a homotopy equivalence. In fact, $\langle g_0\rangle_H$ is a strong deformation retract of $\Gamma(\cF)$. In particular, their corresponding homotopy groups agree in every degree (after choosing a basepoint for $k\geq1$).
\end{corollary}

\begin{proof}
Since $\mathrm{Met}(M)$ is a convex subset of $\Gamma(S^2T^*M)$, it contracts onto $g_0$ by
\begin{equation*}
 g\longmapsto g_t=(1-t)g+tg_0,\qquad 0\leq t\leq1.
\end{equation*}
Using \eqref{eq:metric_product_map}--\eqref{eq:metric_product_inverse}, define
\begin{equation*}
 R_t(\xi)=B_{g_t}\cdot\bigl(B_{g_\xi}^{-1}\cdot\xi\bigr),
 \qquad g_t=(1-t)g_\xi+tg_0.
\end{equation*}
Then $R_0(\xi)=\xi$ and $R_1(\xi)\in\langle g_0\rangle_H$. If $\xi$ already induces $g_0$, then $B_{g_\xi}=B_{g_t}=\Id$, so $R_t(\xi)=\xi$ for every $t$. Thus $(R_t)$ is a strong deformation retraction, and the claimed homotopy equivalence follows.
\end{proof}

\subsection{Parallelizable manifolds and mapping spaces}\label{subsec:parallelizable_mapping_spaces}

Theorem \ref{thm: homotopy_general} and Corollary \ref{cor:metric_map_homotopy_equivalence} show that, from the homotopical point of view, the study of the full space $\Gamma(\mathcal F)$ of $H$-structures reduces to the study of any fixed isometric class $\langle g\rangle_H$. By the discussion in the preliminaries, the latter is naturally identified with the section space
$\Gamma(\mathrm{Fr}(M,g)/H)$, where the fiber of the underlying bundle is the homogeneous space $\mathrm{SO}(n)/H$. Thus, the topology of the space of $H$-structures is ultimately governed by the topology of the corresponding homogeneous section spaces.

In certain geometrically distinguished situations, these section spaces admit a particularly explicit description. For instance, if $(M,g)$ is a flat torus, then the Levi--Civita connection admits a global parallel oriented orthonormal frame, and hence the orthonormal frame bundle is trivial after choosing an oriented parallel orthonormal frame:
\begin{equation*}
\mathrm{Fr}(M,g)\cong M\times \mathrm{SO}(n).
\end{equation*}
Consequently,
\begin{equation*}
\mathrm{Fr}(M,g)/H
\cong
M\times \mathrm{SO}(n)/H,
\end{equation*}
and therefore
\begin{equation*}
\langle g\rangle_H
\cong
C^\infty(M,\mathrm{SO}(n)/H).
\end{equation*}
In particular,
\begin{equation*}
\pi_0(\langle g\rangle_H)
\cong
[M,\mathrm{SO}(n)/H].
\end{equation*}

More generally, the same conclusion holds whenever $M$ is parallelizable. Indeed, given any Riemannian metric $g$ on a parallelizable manifold $M$, a global frame may be orthonormalized via the Gram--Schmidt process, yielding a trivialization
\begin{equation*}
\mathrm{Fr}(M,g)\cong M\times \mathrm{SO}(n).
\end{equation*}
Hence, for parallelizable manifolds, the topology of the space of $H$-structures is reduced to the topology of mapping spaces into the homogeneous space $\mathrm{SO}(n)/H$. Important examples of parallelizable manifolds include Lie groups, products of parallelizable manifolds, all orientable $3$-manifolds, the spheres $S^1$, $S^3$, and $S^7$, and compact quotients of Lie groups by lattices, such as nilmanifolds and solvmanifolds. This identification depends on the choice of a global orthonormal frame. For a flat torus, the flat metric admits a global parallel orthonormal frame, so the trivialization is geometrically preferred. On a general parallelizable manifold, the trivialization is noncanonical, but it gives the same description of the isometric class for every Riemannian metric.

\subsection{Examples on flat tori}

We now illustrate these observations in several classical geometries on flat tori. Since flat tori are parallelizable and their flat metrics admit global parallel orthonormal frames, the corresponding isometric classes reduce to mapping spaces into the homogeneous spaces $\mathrm{SO}(n)/H$. Consequently, their connected components are described by free homotopy classes
\begin{equation*}
[T^n,\mathrm{SO}(n)/H].
\end{equation*}

In the examples below, infinitely many components are distinguished by integral cohomology classes. The relevant cohomology groups of the torus are
\begin{equation*}
H^k(T^n;\mathbb Z)\cong \mathbb Z^{\binom{n}{k}}.
\end{equation*}
The maps constructed below have distinct cohomological invariants, and hence represent distinct elements of $[T^n,\mathrm{SO}(n)/H]$.

We begin with $\mathrm G_2$-structures on $T^7$, where the relevant homotopy classes are detected by the induced map on fundamental groups and by an integral cohomology class in degree seven.

\begin{example}[$\mathrm G_2$-structures on $T^7$]
For $H=\mathrm G_2\subset \mathrm{SO}(7)$, one has
\begin{equation*}
\mathrm{SO}(7)/\mathrm G_2\cong \mathbb{RP}^7.
\end{equation*}
Therefore
\begin{equation*}
\pi_0(\langle g_{\mathrm{flat}}\rangle_{\mathrm G_2})
\cong
[T^7,\mathbb{RP}^7],
\end{equation*}
where the right-hand side denotes free homotopy classes of maps.

Since
\begin{equation*}
\pi_1(T^7)\cong\mathbb Z^7
\qquad\text{and}\qquad
\pi_1(\mathbb{RP}^7)\cong\mathbb Z_2,
\end{equation*}
every map
\begin{equation*}
f:T^7\longrightarrow \mathbb{RP}^7
\end{equation*}
induces a homomorphism
\begin{equation*}
f_*:\mathbb Z^7\longrightarrow \mathbb Z_2.
\end{equation*}
Because $\mathbb Z_2$ is abelian, this induced homomorphism depends only on the free homotopy class of $f$. Hence the assignment
\begin{equation*}
[f]\longmapsto f_*
\end{equation*}
detects at least
\begin{equation*}
\mathrm{Hom}(\mathbb Z^7,\mathbb Z_2)
\cong
(\mathbb Z_2)^7
\end{equation*}
distinct connected components. Equivalently, these connected components are detected by the cohomology class
\begin{equation*}
f^*w\in H^1(T^7;\mathbb Z_2),
\end{equation*}
where $w$ denotes the generator of
\begin{equation*}
H^1(\mathbb{RP}^7;\mathbb Z_2).
\end{equation*}
Indeed, the induced homomorphism
\begin{equation*}
f_*:\pi_1(T^7)\to\pi_1(\mathbb{RP}^7)
\end{equation*}
corresponds, under the identification
\begin{equation*}
H^1(T^7;\mathbb Z_2)
\cong
\mathrm{Hom}(\pi_1(T^7),\mathbb Z_2),
\end{equation*}
precisely to the class $f^*w$.

A map $f:T^7\to\mathbb{RP}^7$ induces the trivial homomorphism on fundamental groups precisely when it lifts to the universal cover
\begin{equation*}
S^7\longrightarrow \mathbb{RP}^7.
\end{equation*}
To construct infinitely many homotopy classes of this type, let $p:S^7\to\mathbb{RP}^7$ be the universal covering and let $u\in H^7(\mathbb{RP}^7;\mathbb Z)$ be an orientation class. Since $p$ has degree two, $p^*u=2v$, where $v$ generates $H^7(S^7;\mathbb Z)$. For each integer $j\geq1$, choose a map $d_j:T^7\to S^7$ of degree $j$. Let $v_7\in H^7(T^7;\mathbb Z)$ be its orientation class. The maps $p\circ d_j$ induce the trivial homomorphism on fundamental groups and satisfy
\begin{equation*}
 (p\circ d_j)^*u=2j\,v_7\in H^7(T^7;\mathbb Z).
\end{equation*}
As these cohomology classes are distinct, the maps $p\circ d_j$ are pairwise nonhomotopic. Consequently, infinitely many homotopy classes induce the trivial homomorphism on fundamental groups, and therefore
\begin{equation*}
\#\pi_0(\langle g_{\mathrm{flat}}\rangle_{\mathrm G_2})=\infty.
\end{equation*}
\end{example}

\begin{example}[$\mathrm{Spin}(7)$-structures on $T^8$]
For $H=\mathrm{Spin}(7)\subset \mathrm{SO}(8)$, one has
\begin{equation*}
\mathrm{SO}(8)/\mathrm{Spin}(7)\cong \mathbb{RP}^7.
\end{equation*}
Therefore
\begin{equation*}
\pi_0(\langle g_{\mathrm{flat}}\rangle_{\mathrm{Spin}(7)})
\cong
[T^8,\mathbb{RP}^7],
\end{equation*}
where the right-hand side denotes free homotopy classes of maps.

Since
\begin{equation*}
\pi_1(T^8)\cong\mathbb Z^8
\qquad\text{and}\qquad
\pi_1(\mathbb{RP}^7)\cong\mathbb Z_2,
\end{equation*}
every map
\begin{equation*}
f:T^8\longrightarrow \mathbb{RP}^7
\end{equation*}
induces a homomorphism
\begin{equation*}
f_*:\mathbb Z^8\longrightarrow \mathbb Z_2.
\end{equation*}
Because $\mathbb Z_2$ is abelian, this induced homomorphism depends only on the free homotopy class of $f$. Hence the assignment
\begin{equation*}
[f]\longmapsto f_*
\end{equation*}
detects at least
\begin{equation*}
\mathrm{Hom}(\mathbb Z^8,\mathbb Z_2)
\cong
(\mathbb Z_2)^8
\end{equation*}
distinct connected components. Equivalently, these components are detected by the cohomology class
\begin{equation*}
f^*w\in H^1(T^8;\mathbb Z_2),
\end{equation*}
where $w$ denotes the generator of
\begin{equation*}
H^1(\mathbb{RP}^7;\mathbb Z_2).
\end{equation*}

In fact, the space has infinitely many connected components. By the lifting criterion, a map $f:T^8\to\mathbb{RP}^7$ induces the trivial homomorphism on fundamental groups precisely when it lifts to the universal cover
\begin{equation*}
S^7\longrightarrow \mathbb{RP}^7.
\end{equation*}
Let $p:S^7\to\mathbb{RP}^7$ denote the universal covering, and let $u\in H^7(\mathbb{RP}^7;\mathbb Z)$ be an orientation class. Let $\rho:T^8\to T^7$ be projection onto seven circle factors, and choose maps $d_j:T^7\to S^7$ of degree $j$. Let $v_7\in H^7(T^7;\mathbb Z)$ be the orientation class. The maps $p\circ d_j\circ\rho$ lift to $S^7$ and satisfy
\begin{equation*}
 (p\circ d_j\circ\rho)^*u=2j\,\rho^*v_7\in H^7(T^8;\mathbb Z).
\end{equation*}
These classes are pairwise distinct, because $\rho^*$ is injective in degree seven. Thus the maps $p\circ d_j\circ\rho$ represent infinitely many homotopy classes with trivial induced map on fundamental groups. Consequently,
\begin{equation*}
\#\pi_0(\langle g_{\mathrm{flat}}\rangle_{\mathrm{Spin}(7)})=\infty.
\end{equation*}
\end{example}

\begin{example}[Almost Hermitian structures on $T^4$]
For $H=\mathrm{U}(2)\subset \mathrm{SO}(4)$, one has
\begin{equation*}
\mathrm{SO}(4)/\mathrm{U}(2)\cong S^2.
\end{equation*}
Therefore
\begin{equation*}
\pi_0(\langle g_{\mathrm{flat}}\rangle_{\mathrm{U}(2)})
\cong
[T^4,S^2].
\end{equation*}

Since
\begin{equation*}
\pi_2(S^2)\cong\mathbb Z,
\end{equation*}
every map
\begin{equation*}
f:T^4\longrightarrow S^2
\end{equation*}
determines a primary integral invariant given by the pullback of the fundamental class,
\begin{equation*}
f^*[S^2]\in H^2(T^4;\mathbb Z).
\end{equation*}
This invariant can take infinitely many values. Indeed, let $\rho:T^4\to T^2$ be projection onto two circle factors, and choose maps $d_j:T^2\to S^2$ of degree $j\geq1$. Write $v_2\in H^2(T^2;\mathbb Z)$ for its orientation class. For $f_j=d_j\circ\rho$, the classes
\begin{equation*}
 f_j^*[S^2]=j\,\rho^*v_2\in H^2(T^4;\mathbb Z)
\end{equation*}
are pairwise distinct, since $\rho^*$ is injective. Consequently,
\begin{equation*}
\#\pi_0(\langle g_{\mathrm{flat}}\rangle_{\mathrm{U}(2)})=\infty.
\end{equation*}

Equivalently, even within the fixed flat metric class, the space of almost Hermitian structures on $T^4$ already has infinitely many connected components.
\end{example}

\begin{example}[$\mathrm{SU}(2)$-structures on $T^4$]\label{ex:SU2_torus}
For $H=\mathrm{SU}(2)\subset \mathrm{SO}(4)$, one has
\begin{equation*}
\mathrm{SO}(4)/\mathrm{SU}(2)\cong \mathbb{RP}^3.
\end{equation*}
Thus
\begin{equation*}
\pi_0(\langle g_{\mathrm{flat}}\rangle_{\mathrm{SU}(2)})
\cong
[T^4,\mathbb{RP}^3].
\end{equation*}

As in the $\mathrm G_2$ and $\mathrm{Spin}(7)$ examples, the fundamental group already detects several components. Since
\begin{equation*}
\pi_1(\mathbb{RP}^3)\cong \mathbb Z_2,
\end{equation*}
there are at least
\begin{equation*}
\mathrm{Hom}(\mathbb Z^4,\mathbb Z_2)
\cong
(\mathbb Z_2)^4
\end{equation*}
distinct classes of induced homomorphisms, equivalently detected by
\begin{equation*}
H^1(T^4;\mathbb Z_2)\cong(\mathbb Z_2)^4.
\end{equation*}

Moreover, maps with trivial induced homomorphism on fundamental groups lift to the universal cover
\begin{equation*}
S^3\longrightarrow \mathbb{RP}^3.
\end{equation*}
Let $p:S^3\to\mathbb{RP}^3$ be the universal double covering, $u\in H^3(\mathbb{RP}^3;\mathbb Z)$ an orientation class, and $\rho:T^4\to T^3$ projection onto three factors. Let $v_3\in H^3(T^3;\mathbb Z)$ be the orientation class. Choosing maps $d_j:T^3\to S^3$ of degree $j\geq1$, we obtain
\begin{equation*}
 (p\circ d_j\circ\rho)^*u=2j\,\rho^*v_3\in H^3(T^4;\mathbb Z).
\end{equation*}
These classes are pairwise distinct, so the corresponding maps $T^4\to\mathbb{RP}^3$ are not homotopic.
Hence
\begin{equation*}
\#\pi_0(\langle g_{\mathrm{flat}}\rangle_{\mathrm{SU}(2)})=\infty.
\end{equation*}
\end{example}

\begin{example}[$\mathrm{SU}(m)$-structures on $T^{2m}$]\label{ex:SU_m_tori}
Put $X_m=\SO(2m)/\SU(m)$, for $m\geq2$. The homotopy exact sequence of
\begin{equation}\label{eq:SU_m_principal_fibration}
 \SU(m)\longrightarrow\SO(2m)\longrightarrow X_m
\end{equation}
gives $\pi_1(X_m)\cong\mathbb Z_2$, since $\SU(m)$ is simply connected. In particular, $H^1(X_m;\mathbb Z)=0$. The homotopy classes of maps $T^{2m}\to X_m$ behave differently in the following cases.

For $m=2$, we have $X_2\cong\mathbb{RP}^3$, and Example~\ref{ex:SU2_torus} shows that $[T^4,X_2]$ is infinite.

For $m=3$, the double covering $\mathrm{Spin}(6)\cong\SU(4)\to\SO(6)$ induces a double covering
\begin{equation*}
 \SU(4)/\SU(3)\cong S^7\longrightarrow X_3.
\end{equation*}
Its deck transformation is the antipodal map, so $X_3\cong\mathbb{RP}^7$. For any map $T^6\to\mathbb{RP}^7$, its induced homomorphism on fundamental groups defines a class in $H^1(T^6;\mathbb Z_2)$. Conversely, every such class is realized by a map $T^6\to S^1\to\mathbb{RP}^7$, since $H^1(T^6;\mathbb Z)\to H^1(T^6;\mathbb Z_2)$ is surjective and a nontrivial loop in $\mathbb{RP}^7$ represents the generator of $\pi_1(\mathbb{RP}^7)$. Two maps inducing the same class are homotopic, because the universal cover $S^7$ is $6$-connected and $T^6$ is a $6$-dimensional CW complex. Consequently,
\begin{equation*}
 [T^6,X_3]\cong H^1(T^6;\mathbb Z_2)\cong(\mathbb Z_2)^6,
\end{equation*}
and $\langle g_{\mathrm{flat}}\rangle_{\SU(3)}$ has exactly $64$ path-connected components.

Now let $m\geq4$. By the homotopy exact sequence of \eqref{eq:SU_m_principal_fibration} and the stable homotopy groups of the classical groups \cite{Bott1959},
\begin{equation*}
 \pi_6(\SO(2m))=0,\qquad
 \pi_5(\SU(m))\cong\mathbb Z,\qquad
 \pi_5(\SO(2m))=0.
\end{equation*}
The relevant portion of the exact sequence therefore reduces to
\begin{equation*}
 0\longrightarrow\pi_6(X_m)\longrightarrow\mathbb Z
 \longrightarrow0,
\end{equation*}
and hence $\pi_6(X_m)\cong\mathbb Z$.

We next show that this homotopy group gives a nonzero rational cohomology class in degree six. Let
\begin{equation*}
 \widetilde X_m=\mathrm{Spin}(2m)/\SU(m)
\end{equation*}
be the simply connected double cover of $X_m$. The map $\pi_3(\SU(m))\otimes\mathbb Q\to\pi_3(\mathrm{Spin}(2m))\otimes\mathbb Q$ induced by the standard real representation is an isomorphism (it is multiplication by a nonzero integer before rationalization). Both groups have vanishing rational homotopy in degrees $2$ and $4$, while $\pi_5(\SU(m))\otimes\mathbb Q\cong\mathbb Q$ and $\pi_5(\mathrm{Spin}(2m))\otimes\mathbb Q=0$. The homotopy exact sequence shows that
\begin{equation*}
 \pi_i(\widetilde X_m)\otimes\mathbb Q=0\quad(1\leq i\leq5),
 \qquad \pi_6(\widetilde X_m)\otimes\mathbb Q\cong\mathbb Q.
\end{equation*}
The rational Hurewicz theorem therefore gives
\begin{equation*}
 H^6(\widetilde X_m;\mathbb Q)\cong\mathbb Q.
\end{equation*}
The nontrivial deck transformation is induced by multiplication by the central element $-1\in\mathrm{Spin}(2m)$. Since $\mathrm{Spin}(2m)$ is connected, this transformation is homotopic to the identity on $\widetilde X_m$. In particular, it acts trivially on rational cohomology. Transfer for the double covering now implies
\begin{equation*}
 H^6(X_m;\mathbb Q)\cong\mathbb Q.
\end{equation*}
Moreover, a generator of $\pi_6(X_m)\otimes\mathbb Q$ pairs nontrivially with a class $u\in H^6(X_m;\mathbb Q)$. Choose $q:S^6\to X_m$ with $q^*u\ne0$.

Let $p:T^{2m}\to T^6$ be projection onto six circle factors, and let $d_j:T^6\to S^6$ be a map of degree $j$, for $j\geq1$. The compositions
\begin{equation*}
 f_j=q\circ d_j\circ p:T^{2m}\longrightarrow X_m
\end{equation*}
satisfy
\begin{equation*}
 f_j^*u=j\,p^*d_1^*q^*u\in H^6(T^{2m};\mathbb Q).
\end{equation*}
Since $p^*$ is injective in degree six and $q^*u\ne0$, these classes are pairwise distinct. Hence the maps $f_j$ are not homotopic. We conclude that
\begin{equation*}
 \#\pi_0(\langle g_{\mathrm{flat}}\rangle_{\SU(m)})=\infty
 \qquad (m=2\ \text{or}\ m\geq4),
\end{equation*}
whereas the corresponding number for $m=3$ is $64$. By Section~\ref{subsec:parallelizable_mapping_spaces}, these counts hold for every Riemannian metric on $T^{2m}$.
\end{example}

The examples show that fixed isometric classes can have infinitely many connected components. This need not hold for every group $H$, as the $\SU(3)$ case demonstrates.

\subsection{Moduli spaces}\label{subsec: moduli}

The action of $\Diff^+(M)$ on $H$-structures by pullback defines the moduli space
\begin{equation}\label{eq:full_moduli}
 \mathcal M_H(M):=\Gamma(\cF)/\Diff^+(M).
\end{equation}
For a fixed metric $g$, the subgroup preserving its isometric class is determined by the identity
\begin{equation*}
 g_{\varphi^*\xi}=\varphi^*g_\xi,
\end{equation*}
so $\varphi^*$ maps $\langle g\rangle_H$ to $\langle\varphi^*g\rangle_H$ and preserves $\langle g\rangle_H$ only if $\varphi$ is an isometry of $g$. Therefore the natural \emph{fixed-metric moduli space} is
\begin{equation}\label{eq:fixed_metric_moduli}
 \mathcal M_H(M,g):=\langle g\rangle_H/\Isom^+(M,g).
\end{equation}
The deformation retraction in Corollary~\ref{cor:metric_map_homotopy_equivalence} is not, in general, $\Diff^+(M)$-equivariant, so it does not identify the homotopy types of these quotients.

For flat tori, the topology of the fixed-metric quotient follows directly from that of the isometric class.

\begin{proposition}\label{prop:flat_torus_isometric_moduli}
Let $g_{\mathrm{flat}}$ be a flat metric on $T^n$. If $\langle g_{\mathrm{flat}}\rangle_H$ has infinitely many path-connected components, then so does $\mathcal M_H(T^n,g_{\mathrm{flat}})$.
\end{proposition}

\begin{proof}
Write $T^n=\mathbb R^n/\Lambda$. Every isometry of the flat torus lifts to an affine isometry of $\mathbb R^n$ preserving $\Lambda$. Hence
\begin{equation*}
 \Isom^+(T^n,g_{\mathrm{flat}})
 \cong T^n\rtimes K,
\end{equation*}
where $K$ is the finite group of orientation-preserving orthogonal transformations preserving $\Lambda$. The torus of translations is the identity component of the isometry group. Since each translation is joined to the identity by a path of translations, its action on $\langle g_{\mathrm{flat}}\rangle_H$ preserves every path component. The remaining group $K$ has only finitely many elements, so its orbits on the set of path components are finite.

As $T^n$ is compact and $\SO(n)/H$ is a compact manifold, the path components of $C^\infty(T^n,\SO(n)/H)$ are open in the $C^\infty$ topology: sufficiently close maps are joined by the pointwise short-geodesic homotopy. Thus the components of $\langle g_{\mathrm{flat}}\rangle_H$ are open. Their images in the quotient form disjoint open path-connected subsets, indexed by the $K$-orbits on the set of components. There are infinitely many such orbits, which proves the proposition.
\end{proof}

The induced action of diffeomorphisms on first homology gives a surjection
\begin{equation*}
 \pi_0(\Diff^+(T^n))\longrightarrow\SL(n,\mathbb Z),
\end{equation*}
but in high dimensions this map is not generally an isomorphism \cite{Bustamante2024}. Moreover, after identifying a fixed isometric class with $C^\infty(T^n,\SO(n)/H)$ by means of a global frame, pulling back the corresponding tensor field involves both precomposition and the differential of the diffeomorphism. Therefore the action is not, in general, described solely by the usual precomposition action on mapping spaces.

One can nevertheless obtain infinitely many components of the full moduli space when the structures carry suitable characteristic classes. We illustrate this for almost Hermitian structures on $T^4$.

\begin{proposition}\label{prop:U2_full_moduli}
The moduli space $\mathcal M_{\U(2)}(T^4)$ has infinitely many path-connected components.
\end{proposition}

\begin{proof}
Fix the standard orientation and a flat metric on $T^4$, together with an oriented parallel orthonormal frame. This identifies the space of compatible almost complex structures with $C^\infty(T^4,\SO(4)/\U(2))=C^\infty(T^4,S^2)$. For a smooth map $f:T^4\to S^2$, write $J_f$ for the corresponding almost complex structure on the \emph{trivial real bundle} $TT^4$.

We first determine its first Chern class. Consider the universal complex rank-two vector bundle over $S^2=\SO(4)/\U(2)$ obtained by equipping the trivial real four-plane bundle with the complex structure specified by the point of $S^2$. Let $u$ denote a generator of $H^2(S^2;\mathbb Z)$. Its first Chern class is $\pm2u$. Indeed, the homotopy exact sequence of
\begin{equation*}
 \U(2)\longrightarrow\SO(4)\longrightarrow S^2
\end{equation*}
contains
\begin{equation*}
 0=\pi_2(\SO(4))\longrightarrow\pi_2(S^2)\cong\mathbb Z
 \xrightarrow{\ \partial\ }\pi_1(\U(2))\cong\mathbb Z
 \longrightarrow\pi_1(\SO(4))\cong\mathbb Z_2.
\end{equation*}
The last map is reduction modulo two, so $\partial$ identifies $\pi_2(S^2)$ with $2\mathbb Z$. Since the determinant induces an isomorphism $\pi_1(\U(2))\cong\pi_1(\U(1))$, the clutching construction gives first Chern number $\pm2$. Pulling back the universal bundle along $f$ yields
\begin{equation}\label{eq:chern_class_U2}
 c_1(TT^4,J_f)=\pm2 f^*u.
\end{equation}

Let $p:T^4\to T^2$ be projection onto two circle factors and choose a map $d_j:T^2\to S^2$ of degree $j$, for each $j\geq1$. Put $f_j=d_j\circ p$. Since $p^*H^2(T^2;\mathbb Z)$ is a primitive direct summand of $H^2(T^4;\mathbb Z)$, the class $f_j^*u$ has divisibility $j$. By \eqref{eq:chern_class_U2}, $c_1(TT^4,J_{f_j})$ has divisibility $2j$.

Finally, the first Chern class is constant under homotopies of almost complex structures and is natural under diffeomorphisms:
\begin{equation*}
 c_1(TT^4,\varphi^*J)=\varphi^*c_1(TT^4,J).
\end{equation*}
Every automorphism of the free abelian group $H^2(T^4;\mathbb Z)$ preserves the divisibility of its elements. The divisibility of $c_1(TT^4,J)$ is locally constant in the $C^\infty$ topology on the space of almost complex structures, and it is invariant under $\Diff^+(T^4)$. It therefore defines a continuous, integer-valued function on the quotient $\mathcal M_{\U(2)}(T^4)$. As this function takes the distinct values $2j$ on the images of $J_{f_j}$, these images belong to pairwise distinct path components.
\end{proof}

\begin{remark}
More generally, the divisibility of an integral cohomology class is preserved by every diffeomorphism, even though a diffeomorphism need not fix the class itself. This observation distinguishes infinitely many orbits for the precomposition action on various mapping spaces. For geometric structures, the differential of a diffeomorphism also acts on a chosen framing. Proposition~\ref{prop:U2_full_moduli} instead uses the intrinsic first Chern class. The fixed-metric statement of Proposition~\ref{prop:flat_torus_isometric_moduli} applies, independently, to all of the infinite-component examples in the preceding subsection.
\end{remark}

\subsection{Topological interpretation of finite-time singularity examples}\label{subsec: finite-time_sing}

The results above also clarify the finite-time singularity examples in \cite[Example~2.17 and Remark~2.18]{Fadel2022}. Their key feature is the existence, in a fixed isometric homotopy class, of structures whose intrinsic torsion energy tends to zero, although that class contains no torsion-free structure. On a flat torus, this is closely related to concentrating a nontrivial homotopy class into a small ball. Throughout this subsection $\mathcal E(\xi)=\frac12\int_{T^n}|T_\xi|^2\,\mathrm{vol}_{g_{\mathrm{flat}}}$ denotes the intrinsic torsion energy on the fixed isometric class.

Let $(T^n,g_{\mathrm{flat}})$ be a flat torus and put $X=\SO(n)/H$. A choice of parallel orthonormal frame identifies
\begin{equation*}
 \langle g_{\mathrm{flat}}\rangle_H\cong C^\infty(T^n,X).
\end{equation*}
If $B\subset T^n$ is an embedded open ball, collapsing its complement produces a degree-one map
\begin{equation*}
 p_B:T^n\longrightarrow T^n/(T^n\setminus B)\cong S^n.
\end{equation*}
Given a based map $q:S^n\to X$, the composition $q\circ p_B$ is constant outside $B$ (after a harmless smoothing near its boundary). The class represented on the torus is $p_B^*[q]\in[T^n,X]$.

For the concentration construction, we require that $q\circ p_B$ represent a nontrivial element of $[T^n,X]$. This condition is stronger than requiring $[q]\ne0$ in $\pi_n(X)$.

\begin{proposition}\label{prop:pinching_zero_infimum}
Let $n>2$ and suppose that $q:S^n\to X=\SO(n)/H$ is a smooth based map such that $q\circ p_B:T^n\to X$ is not null-homotopic. Then the corresponding nontrivial isometric homotopy class $[\xi]$ of compatible $H$-structures on $(T^n,g_{\mathrm{flat}})$ satisfies
\begin{equation*}
 \inf_{\eta\in[\xi]}\mathcal E(\eta)=0,
\end{equation*}
but contains no torsion-free $H$-structure.
\end{proposition}

\begin{proof}
Choose a smooth representative $q$ which is constant on a neighbourhood of the basepoint of $S^n$. Let $x_0\in T^n$ and choose a flat coordinate ball about $x_0$. For sufficiently small $\varepsilon>0$, the map $q\circ p_B$ can be represented by a smooth map $f_\varepsilon:T^n\to X$ which is constant outside $B_\varepsilon(x_0)$ and is obtained by rescaling a fixed map on the unit ball. All these maps represent the same homotopy class: the collapse maps of embedded balls are degree-one maps $T^n\to S^n$, hence homotopic, and smoothing does not alter their homotopy class.

The fixed parallel frame identifies the intrinsic torsion energy of the associated structure $\xi_\varepsilon$ with a fixed positive multiple of the Dirichlet energy of $f_\varepsilon$. In flat coordinates, a change of variables gives
\begin{equation}\label{eq:pinching_energy_scale}
 \int_{T^n}|df_\varepsilon|^2\,\mathrm{vol}_{g_{\mathrm{flat}}}
 =\varepsilon^{n-2}\int_{B_1(0)}|df_1|^2\,dx.
\end{equation}
Since $n>2$, the right-hand side tends to zero as $\varepsilon\to0$. Thus $\inf_{[\xi]}\mathcal E=0$.

On the other hand, a torsion-free compatible $H$-structure on the flat torus is parallel. Relative to the chosen parallel frame, its associated map $T^n\to X$ is constant and hence null-homotopic. By hypothesis, this is not the class of $f_\varepsilon$. Therefore $[\xi]$ contains no torsion-free structure.
\end{proof}

\begin{example}\label{ex:G2_pinching}
For $n=7$ and $H=\mathrm G_2$, the target is $X\cong\mathbb{RP}^7$. Take $q:S^7\to\mathbb{RP}^7$ to be the universal double covering, and let $u\in H^7(\mathbb{RP}^7;\mathbb Z)$ be an orientation class. Then $q^*u$ is twice a generator of $H^7(S^7;\mathbb Z)$, so
\begin{equation*}
 (q\circ p_B)^*u=2\,\nu\ne0\quad\text{in }H^7(T^7;\mathbb Z),
\end{equation*}
where $\nu\in H^7(T^7;\mathbb Z)$ denotes the cohomological orientation class. This verifies directly that $q\circ p_B$ is nontrivial, and Proposition~\ref{prop:pinching_zero_infimum} applies.
\end{example}

The nontrivial homotopy class excludes a torsion-free representative, while the supercritical scaling \eqref{eq:pinching_energy_scale} permits the energy to approach zero. Under the hypotheses of \cite[Theorem~2.16]{Fadel2022}, one consequently obtains initial structures with arbitrarily small intrinsic torsion energy whose harmonic $H$-flow develops a finite-time singularity. The required analytical estimates are discussed further in Section~\ref{sec:postscriptum_flow_extensions}.

This is the same concentration mechanism encountered in the harmonic map heat flow in dimensions greater than two. It also makes clear why not every nontrivial component of an isometric class has vanishing energy infimum. The next section gives a contrasting family of cohomological classes for which the energy has a positive lower bound.

\section{Variational consequences of the topology of $H$-structures}
\label{sec:variational_consequences_topology}

Theorem~\ref{thm: homotopy_general} and Corollary~\ref{cor:metric_map_homotopy_equivalence} show that, up to homotopy, the topology of the full space of $H$-structures is already encoded by fixed isometric classes. We now explain how, on closed manifolds, this viewpoint interacts with the variational theory of the intrinsic torsion energy. There are two complementary phenomena. On the unrestricted space of $H$-structures, allowing the induced metric to vary introduces an artificial scaling direction. By contrast, inside a fixed isometric class this scaling direction is absent, and the topology of the class leads to genuinely variational questions, including concentration mechanisms, positive lower bounds, and the existence of nontrivial harmonic representatives.

\subsection{The unrestricted torsion energy}
\label{subsec:unrestricted_torsion_energy}

We first make precise the scale-degeneracy of the unrestricted functional. Starting from any $H$-structure $\xi_0$, one may act by the homothetic bundle automorphisms $e^{-t}\mathrm{Id}_{TM}$, obtaining a smooth path of $H$-structures whose induced metrics are $e^{-2t}g_{\xi_0}$. Along this explicit path, which exists for every finite $t\geqslant0$, the intrinsic torsion is unchanged as an $\mathfrak m$-valued one-form, while its total $L^2$-energy scales by the factor $e^{-(n-2)t}$. Thus, in dimensions $n>2$, the energy infimum is zero on every nonempty path component of the unrestricted space. Moreover, differentiating this same variation at $t=0$ shows that any unrestricted critical point must have zero torsion energy. This gives a uniform explanation, for arbitrary closed and connected subgroups $H\leqslant\SO(n)$, of why unrestricted critical points are precisely torsion-free structures.

\begin{theorem}[Scale degeneracy of the unrestricted torsion energy]
\label{thm:unrestricted_torsion_energy_degeneracy}
Let $M^n$ be compact, with $n>2$, and let
\begin{equation*}
\mathcal E(\xi)
=
\frac12\int_M |T_\xi|^2\,\mathrm{vol}_{g_\xi}
\end{equation*}
be the intrinsic torsion energy on the full space $\Gamma(\mathcal F)$ of $H$-structures, where the induced metric is allowed to vary. Then the following hold.

\begin{enumerate}
\item For every $\xi_0\in\Gamma(\mathcal F)$ and every $t\geqslant0$, there exists a smooth path $\xi_t\in\Gamma(\mathcal F)$ starting at $\xi_0$ such that
\begin{equation*}
g_{\xi_t}=e^{-2t}g_{\xi_0}
\end{equation*}
and
\begin{equation*}
\mathcal E(\xi_t)
=
e^{-(n-2)t}\mathcal E(\xi_0).
\end{equation*}

\item In particular, the infimum of $\mathcal E$ on every nonempty path component of $\Gamma(\mathcal F)$ is zero.

\item The critical points of $\mathcal E$ on the unrestricted space $\Gamma(\mathcal F)$ are precisely the torsion-free $H$-structures. Equivalently, if $\xi$ is a critical point of $\mathcal E$ on $\Gamma(\mathcal F)$, then $T_\xi=0$.

\item Consequently, if a path component of $\Gamma(\mathcal F)$ contains no torsion-free $H$-structure, then $\mathcal E$ has no critical point on that component and its infimum is not achieved.
\end{enumerate}
\end{theorem}

\begin{proof}
Let $\xi_0\in\Gamma(\mathcal F)$ be an $H$-structure inducing the Riemannian metric $g_0:=g_{\xi_0}$, and let $T_0:=T_{\xi_0}$ denote its intrinsic torsion. Consider the homothetic family of metrics
\begin{equation*}
g_t=e^{-2t}g_0.
\end{equation*}
This is the solution of
\begin{equation*}
\frac{\partial}{\partial t}g_t=-2g_t,
\end{equation*}
and corresponds, in the notation of the metric variation formula, to the symmetric endomorphism $A_t=-\mathrm{Id}$. By Theorem~\ref{thm: homotopy_general}, this metric path is lifted by the family of bundle automorphisms
\begin{equation*}
F_t=e^{-t}\mathrm{Id}_{TM}.
\end{equation*}
Equivalently, the corresponding path of $H$-structures is
\begin{equation*}
\xi_t:=F_t\cdot\xi_0.
\end{equation*}
For every finite $t$, the map $F_t$ is an orientation-preserving bundle automorphism, so $\xi_t$ is a smooth $H$-structure and
\begin{equation*}
g_{\xi_t}=g_t.
\end{equation*}

Since $g_t$ differs from $g_0$ by a constant homothety, the Levi--Civita connection is unchanged. Moreover, $F_t=e^{-t}\mathrm{Id}$ is parallel and commutes with every endomorphism of $TM$. Therefore, if $T_0$ is defined by
\begin{equation*}
\nabla_X\xi_0=(T_0)_X\diamond\xi_0,
\end{equation*}
then
\begin{equation*}
\nabla_X\xi_t
=
\nabla_X(F_t\cdot\xi_0)
=
F_t\cdot\nabla_X\xi_0
=
F_t\cdot\bigl((T_0)_X\diamond\xi_0\bigr)
=
(T_0)_X\diamond(F_t\cdot\xi_0)
=
(T_0)_X\diamond\xi_t.
\end{equation*}
Thus the intrinsic torsion is unchanged as an $\mathfrak m$-valued one-form:
\begin{equation*}
T_{\xi_t}=T_{\xi_0}.
\end{equation*}

The pointwise norm, however, changes because the covector index is measured using the rescaled metric. Since $g_t=e^{-2t}g_0$, one has
\begin{equation*}
|T_{\xi_t}|^2_{g_t}
=
e^{2t}|T_{\xi_0}|^2_{g_0},
\end{equation*}
while the volume form satisfies
\begin{equation*}
\mathrm{vol}_{g_t}
=
e^{-nt}\mathrm{vol}_{g_0}.
\end{equation*}
Consequently,
\begin{equation*}
\mathcal E(\xi_t)
=
\frac12\int_M |T_{\xi_t}|^2_{g_t}\,\mathrm{vol}_{g_t}
=
e^{2t}e^{-nt}
\frac12\int_M |T_{\xi_0}|^2_{g_0}\,\mathrm{vol}_{g_0}
=
e^{-(n-2)t}\mathcal E(\xi_0).
\end{equation*}

Since $n>2$, letting $t\to+\infty$ gives
\begin{equation*}
\mathcal E(\xi_t)\longrightarrow0.
\end{equation*}
As $\mathcal E\geqslant0$, this proves that the infimum of $\mathcal E$ on the path component of $\xi_0$ is zero. Since $\xi_0$ was arbitrary, the same holds on every nonempty path component of $\Gamma(\mathcal F)$.

Now suppose that $\xi_0$ is a critical point of $\mathcal E$ on the full space $\Gamma(\mathcal F)$. The homothetic path above is an admissible variation through $H$-structures, so differentiating the identity for $\mathcal E(\xi_t)$ at $t=0$ gives
\begin{equation*}
0
=
\left.\frac{d}{dt}\right|_{t=0}\mathcal E(\xi_t)
=
-(n-2)\mathcal E(\xi_0).
\end{equation*}
Since $n>2$, it follows that $\mathcal E(\xi_0)=0$. Hence $T_{\xi_0}=0$, so $\xi_0$ is torsion-free.

Conversely, if $T_\xi=0$, then $\mathcal E(\xi)=0$, so $\xi$ is an absolute minimizer of the nonnegative functional $\mathcal E$ and therefore a critical point. Thus the unrestricted critical points are precisely the torsion-free $H$-structures.

Finally, if a path component of $\Gamma(\mathcal F)$ contains no torsion-free $H$-structure, then it contains no critical point of $\mathcal E$. Since the infimum on that component is zero but no element in the component has zero torsion energy, the infimum is not achieved.
\end{proof}

\begin{remark}
The conclusion that unrestricted critical points of the torsion energy are torsion-free had previously been observed in special geometries. In the $\mathrm G_2$ case, Weiss--Witt proved that the critical points of the Dirichlet energy on positive $3$-forms are precisely the torsion-free positive $3$-forms, using the positive homogeneity of the functional under scaling of the defining form \cite[Corollary~4.3]{WeissWitt2012}. In the $\mathrm{Spin}(7)$ case, Dwivedi similarly observed that the torsion energy is positively homogeneous under the scaling $\Phi\mapsto c^4\Phi$, and applied Euler's theorem for homogeneous functionals to conclude that its critical points are torsion-free and absolute minimizers \cite[Remark~3.6]{Dwivedi2025gradient}. 

The argument above gives a uniform proof for arbitrary closed and connected subgroups $H\leqslant\SO(n)$ in dimensions $n>2$. Moreover, it also records the stronger scale-degeneracy statement: on every nonempty path component of the unrestricted space of $H$-structures, the infimum of $\mathcal E$ is zero, and if that component contains no torsion-free structure, then the infimum is not achieved and no critical point exists there. Thus the unrestricted variational problem is degenerate because the induced metric may collapse homothetically. The fixed-isometric-class restriction removes this artificial scaling direction and is therefore essential for a meaningful variational theory with nontrivial harmonic representatives.
\end{remark}

\subsection{Cohomological classes and harmonic representatives}
\label{subsec:cohomological_classes_harmonic_representatives}

We now pass to the fixed-isometric-class setting, where the homothetic collapse direction exhibited above is no longer available. In this setting, the topology of the class becomes genuinely relevant to the variational problem. Section~\ref{subsec: finite-time_sing} recalled the finite-time singularity examples of \cite[Example~2.17 and Remark~2.18]{Fadel2022} and interpreted them, using Theorem~\ref{thm: homotopy_general} and Corollary~\ref{cor:metric_map_homotopy_equivalence}, as concentration phenomena in nontrivial isometric homotopy classes with vanishing energy infimum.

The mechanism behind those examples concerns maps $q:S^n\to\SO(n)/H$ for which the composition with the collapse map
\begin{equation*}
T^n\longrightarrow T^n/(T^n\setminus B)\cong S^n
\end{equation*}
represents a nontrivial class in $[T^n,\SO(n)/H]$, where $B\subset T^n$ is an embedded ball. By Proposition~\ref{prop:pinching_zero_infimum}, shrinking $B$ then gives representatives of the same nontrivial class with arbitrarily small energy when $n>2$.

This concentration mechanism does not apply to all nontrivial isometric homotopy classes. For instance, for $\mathrm U(3)$-structures on the flat $6$-torus one has
\begin{equation*}
\SO(6)/\mathrm U(3)\cong\mathbb{CP}^3
\end{equation*}
and
\begin{equation*}
\pi_6(\mathbb{CP}^3)=0.
\end{equation*}
Thus there is no nontrivial top-dimensional sphere class in $\mathbb{CP}^3$ which can be concentrated inside small balls in the manner used in \cite[Example~2.17 and Remark~2.18]{Fadel2022}.

Nevertheless, the corresponding isometric class still has nontrivial topology. Indeed, on the flat torus the identification of fixed isometric classes with mapping spaces gives
\begin{equation*}
\pi_0(\langle g_{\mathrm{flat}}\rangle_{\mathrm U(3)})
\cong
[T^6,\mathbb{CP}^3].
\end{equation*}
These homotopy classes are detected, already at the primary level, by the cohomology class
\begin{equation*}
f^*[\omega]\in H^2(T^6;\mathbb Z),
\end{equation*}
where $[\omega]\in H^2(\mathbb{CP}^3;\mathbb Z)$ denotes the generator.

Such degree-two cohomological classes behave differently from the top-dimensional concentration classes. If $f^*[\omega]\neq0$, then there exists a closed $4$-form $\eta$ on $T^6$ such that
\begin{equation*}
\int_{T^6} f^*\omega\wedge\eta\neq0.
\end{equation*}
On the other hand, since $|f^*\omega|\leqslant C|df|^2$, one obtains
\begin{equation*}
0<
\left|
\int_{T^6} f^*\omega\wedge\eta
\right|
\leqslant
C\|\eta\|_{L^\infty}
\int_{T^6}|df|^2\,\mathrm{vol}_{g_{\mathrm{flat}}}.
\end{equation*}
Hence, the Dirichlet energy has a positive lower bound in any component with nonzero degree-two invariant. Under the fixed flat trivialization of the homogeneous bundle, this Dirichlet energy is equivalent, up to a universal constant depending only on the model, to the intrinsic torsion energy of the associated compatible $\mathrm U(3)$-structure.

Thus, fixed isometric classes exhibit a dichotomy not present in the unrestricted problem. Top-dimensional sphere classes may have zero energy infimum and lead to the finite-time singularity mechanism described in Section~\ref{subsec: finite-time_sing}. By contrast, components detected by lower-degree cohomological invariants, such as the $\mathrm U(3)$-structures on $T^6$, have positive energy lower bounds. These components are therefore naturally related to existence questions for harmonic representatives in nontrivial isometric classes.

We now give a simple construction showing that this latter situation does occur: certain nontrivial cohomological classes contain smooth harmonic representatives.

\begin{proposition}[Holomorphic construction of harmonic $\mathrm U(3)$-structures]
\label{prop:holomorphic_U3_harmonic_structures}
Let $T^6=\mathbb C^3/\Lambda$ be a complex torus equipped with a flat Kähler metric $g_{\mathrm{flat}}$. Suppose that $L\to T^6$ is a holomorphic line bundle admitting four global holomorphic sections
\begin{equation*}
s_0,s_1,s_2,s_3\in H^0(T^6,L)
\end{equation*}
with no common zero. Then these sections define a holomorphic map
\begin{equation*}
f:T^6\longrightarrow \mathbb{CP}^3,
\qquad
f=[s_0:s_1:s_2:s_3].
\end{equation*}
Under the identification
\begin{equation*}
\SO(6)/\mathrm U(3)\cong\mathbb{CP}^3,
\end{equation*}
the map $f$ determines a smooth harmonic $\mathrm U(3)$-structure compatible with $g_{\mathrm{flat}}$.

Moreover,
\begin{equation*}
f^*[\omega_{\mathrm{FS}}]=c_1(L)\in H^2(T^6;\mathbb Z),
\end{equation*}
where $[\omega_{\mathrm{FS}}]$ denotes the positive generator of $H^2(\mathbb{CP}^3;\mathbb Z)$. In particular, if $c_1(L)\neq0$, then the resulting harmonic $\mathrm U(3)$-structure lies in a nontrivial isometric homotopy class.
\end{proposition}

\begin{proof}
Since the sections $s_0,s_1,s_2,s_3$ have no common zero, they define a holomorphic map
\begin{equation*}
f:T^6\longrightarrow \mathbb{CP}^3,
\qquad
f=[s_0:s_1:s_2:s_3].
\end{equation*}

The domain $T^6=\mathbb C^3/\Lambda$ is Kähler with respect to the chosen flat Kähler metric, and $\mathbb{CP}^3$ is Kähler with the Fubini--Study metric. Hence every holomorphic map
\begin{equation*}
f:T^6\longrightarrow \mathbb{CP}^3
\end{equation*}
is harmonic.

On the other hand, the flat metric on $T^6$ gives a global parallel oriented orthonormal frame, and therefore
\begin{equation*}
\mathrm{Fr}(T^6,g_{\mathrm{flat}})
\cong
T^6\times\SO(6).
\end{equation*}
Thus the corresponding bundle of compatible $\mathrm U(3)$-structures is
\begin{equation*}
\mathrm{Fr}(T^6,g_{\mathrm{flat}})/\mathrm U(3)
\cong
T^6\times\SO(6)/\mathrm U(3)
\cong
T^6\times\mathbb{CP}^3.
\end{equation*}
Therefore compatible $\mathrm U(3)$-structures inducing $g_{\mathrm{flat}}$ are identified with smooth maps
\begin{equation*}
T^6\longrightarrow\mathbb{CP}^3.
\end{equation*}

Under this identification, the harmonic section equation is precisely the harmonic map equation for $f$. Hence the $\mathrm U(3)$-structure determined by $f$ is harmonic.

It remains to identify its isometric homotopy class. By the standard construction of maps to projective space from base-point-free linear systems, one has
\begin{equation*}
f^*\mathcal O_{\mathbb{CP}^3}(1)\cong L.
\end{equation*}
Taking first Chern classes gives
\begin{equation*}
f^*c_1(\mathcal O_{\mathbb{CP}^3}(1))=c_1(L).
\end{equation*}
With the standard integral normalization $[\omega_{\mathrm{FS}}]=c_1(\mathcal O_{\mathbb{CP}^3}(1))$, this gives
\begin{equation*}
f^*[\omega_{\mathrm{FS}}]=c_1(L).
\end{equation*}

If $c_1(L)\neq0$, then $f$ cannot be homotopic to a constant map. Therefore the associated $\mathrm U(3)$-structure lies in a nontrivial connected component of
\begin{equation*}
\langle g_{\mathrm{flat}}\rangle_{\mathrm U(3)}
\cong
C^\infty(T^6,\mathbb{CP}^3).
\end{equation*}
\end{proof}

The previous proposition becomes particularly effective on complex tori which are projective, namely on abelian threefolds.

\begin{corollary}
\label{cor:abelian_threefold_harmonic_U3}
Let $T^6=\mathbb C^3/\Lambda$ be an abelian threefold equipped with a flat Kähler metric, and let $L\to T^6$ be an ample holomorphic line bundle. Then, for every $k\geqslant2$, the line bundle $L^{\otimes k}$ determines a smooth harmonic $\mathrm U(3)$-structure in a nontrivial isometric homotopy class.
\end{corollary}

\begin{proof}
For an ample line bundle on an abelian variety, $L^{\otimes k}$ is globally generated for every $k\geqslant2$ \cite{BirkenhakeLange2004}. Moreover, by Riemann--Roch for abelian varieties,
\begin{equation*}
h^0(T^6,L^{\otimes k})=k^3h^0(T^6,L).
\end{equation*}
Since $L$ is ample, $h^0(T^6,L)\geqslant1$, and hence $h^0(T^6,L^{\otimes k})\geqslant8$ for every $k\geqslant2$.

Because $L^{\otimes k}$ is globally generated and $\dim_{\mathbb C}T^6=3$, four general global sections of $L^{\otimes k}$ have no common zero. Indeed, one may choose sections successively so that their common zero sets have codimensions at least $1$, $2$ and $3$; the final section can be chosen nonzero at the remaining finitely many points. This is the usual dimension-count argument for a base-point-free linear system. Applying Proposition~\ref{prop:holomorphic_U3_harmonic_structures} to $L^{\otimes k}$ yields a harmonic $\mathrm U(3)$-structure.

Finally,
\begin{equation*}
c_1(L^{\otimes k})=k\,c_1(L)\neq0,
\end{equation*}
so the resulting harmonic structure lies in a nontrivial isometric homotopy class.
\end{proof}

Proposition \ref{prop:holomorphic_U3_harmonic_structures} and Corollary \ref{cor:abelian_threefold_harmonic_U3} provide a positive counterpart to the finite-time singularity examples. The top-dimensional sphere classes considered in \cite[Example~2.17 and Remark~2.18]{Fadel2022} have zero energy infimum and can force singularity formation under the harmonic flow. By contrast, the cohomological classes detected by $H^2(T^6;\mathbb Z)$ for $\mathrm U(3)$-structures on the flat $6$-torus have positive energy lower bounds, and many of them contain explicit smooth harmonic representatives arising from holomorphic maps to $\mathbb{CP}^3$. This suggests a broader variational theory for harmonic $H$-structures inside nontrivial isometric classes.

%%%%%%%%%%%%%%%%%%%%%%%%%%%%%%%%

\section{Further results on flows of $H$-structures}

We study three aspects of flows of tensorial $H$-structures. First, the metric lifting construction from Section~\ref{sec: topology_of_H_structures} gives a lifting principle for metric-dependent evolution equations, including the Ricci $H$-flow. The lift is obtained by solving a time-dependent linear ODE on bundle automorphisms. We then derive the intrinsic torsion identity \eqref{eq:postscriptum-258-general} for isometric flows. Finally, we use a uniform comparison of $|\nabla\xi|^2$ and $|T|^2$ to extend the harmonic-flow estimates of \cite{Fadel2022} to general closed and connected subgroups $H\leqslant\SO(n)$.

\subsection{Metric-dependent flows and the Ricci $H$-flow}

We consider flows of $H$-structures whose evolution is determined by the induced metric. The Ricci $H$-flow, for which the induced metric evolves by Ricci flow, is a particular example (see \cite[Example~1.26 and Lemma~1.27]{Fadel2022}).

Suppose that the evolution equation for an $H$-structure has the form
\begin{equation}\label{eq:general_metric_flow_H_structure}
\frac{\partial\xi}{\partial t}=S(g_\xi)\diamond\xi,
\qquad \xi(0)=\xi_0,
\end{equation}
where $S(g)$ is a $g$-self-adjoint endomorphism depending on the metric $g$. By \cite[Lemma~1.24]{Fadel2022}, the corresponding metric equation is
\begin{equation}\label{eq:general_metric_flow_metric}
\frac{\partial g}{\partial t}(X,Y)=2g(S(g)X,Y),
\qquad g(0)=g_{\xi_0}.
\end{equation}

\begin{proposition}[Lifting principle for metric-dependent flows]\label{prop:general_lifting_metric_flows}
Suppose that \eqref{eq:general_metric_flow_metric} admits a smooth solution $g(t)$ on $[0,T)$. Then there exists a smooth solution $\xi(t)$ of \eqref{eq:general_metric_flow_H_structure} on $[0,T)$ such that $g_{\xi(t)}=g(t)$.
\end{proposition}
\begin{proof}
Solve the time-dependent linear ODE
\begin{equation*}
\dot F(t)=S(g(t))\circ F(t),\qquad F(0)=\Id,
\end{equation*}
and set $\xi(t)=F(t)\cdot\xi_0$. Then $\dot\xi=S(g(t))\diamond\xi$. Let $q(t)=g_{\xi(t)}$. The metric variation formula \cite[Lemma~1.24]{Fadel2022} gives
\begin{equation*}
\dot q(X,Y)=q(S(g(t))X,Y)+q(X,S(g(t))Y).
\end{equation*}
The prescribed metric $g(t)$ satisfies the same linear equation, because $S(g(t))$ is $g(t)$-self-adjoint. Since $q(0)=g(0)$, uniqueness gives $q(t)=g(t)$. Thus $\xi(t)$ solves \eqref{eq:general_metric_flow_H_structure}.
\end{proof}

The Ricci $H$-flow is the special case $S(g)=-\Ric(g)^\sharp$, where $\Ric(g)^\sharp$ is the Ricci endomorphism.

\begin{corollary}[Short-time existence of Ricci $H$-flow]
\label{lem: Ricci flow}
Let $(M^n,g_0)$ be a closed Riemannian manifold admitting a compatible $H$-structure $\xi_0$. Then there exists $\tau>0$ such that there is a solution $\xi(t)$, defined for all $t\in[0,\tau)$, to the problem
\begin{equation}
\label{eq: Ricci flow_H-str}
\begin{cases}
\displaystyle
\frac{\partial}{\partial t}\xi(t)
=
-\Ric(g_{\xi(t)})^\sharp\diamond\xi(t),
\\[0.3cm]
\xi(0)=\xi_0.
\end{cases}
\end{equation}
\end{corollary}

\begin{proof}
By the short-time existence and uniqueness theorem for Ricci flow, originally due to Hamilton \cite{hamilton1982three} and later simplified by DeTurck \cite{deturck1983deforming}, there exist $\tau>0$ and a unique smooth family of Riemannian metrics $g(t)$, defined for all $t\in[0,\tau)$, satisfying
\begin{equation*}
\begin{cases}
\displaystyle
\frac{\partial}{\partial t}g(t)
=
-2\mathrm{Ric}(g(t)),
\\[0.3cm]
g(0)=g_0=g_{\xi_0}.
\end{cases}
\end{equation*}

The result therefore follows immediately from Proposition~\ref{prop:general_lifting_metric_flows}.
\end{proof}

\begin{remark}\label{rem:ricci_H_flow_lift}
In the proof of \cite[Lemma~1.27]{Fadel2022}, the lift of the Ricci flow was defined by exponentiating the time integral of the Ricci endomorphism. More generally, for a time-dependent endomorphism $S(t)$, the expression
\begin{equation*}
F(t)=\exp\left(\int_0^t S(s)\,ds\right)
\end{equation*}
need not solve $\dot F(t)=S(t)\circ F(t)$. Indeed, if $K(t)=\int_0^t S(s)\,ds$, then
\begin{equation*}
\frac{d}{dt}e^{K(t)}
=\int_0^1 e^{(1-u)K(t)}S(t)e^{uK(t)}\,du,
\end{equation*}
which reduces to $S(t)e^{K(t)}$ when $K(t)$ and $S(t)$ commute, but not in general. The lifting construction in Proposition~\ref{prop:general_lifting_metric_flows} uses instead the evolution operator of the non-autonomous linear equation
\begin{equation*}
\dot F(t)=S(t)\circ F(t),\qquad F(0)=\Id.
\end{equation*}
It therefore establishes the short-time existence asserted in \cite[Lemma~1.27]{Fadel2022} without a commutativity assumption.
\end{remark}

\subsection{An evolution identity for isometric flows}\label{subsec:evolution_identity_isometric_flows}

Let $(M^n,g)$ be a Riemannian manifold equipped with a compatible $H$-structure $\xi$, where $H\leqslant\SO(n)$ is closed and connected. Let $\xi_t$ be a smooth one-parameter family of compatible $H$-structures inducing the fixed metric $g$, and suppose that
\begin{equation*}
    \frac{\partial \xi_t}{\partial t}
    =
    C\diamond \xi_t,
\end{equation*}
where $C=C(t)\in\Omega^2_{\fm}(M)$. We shall prove that, for every vector field $\partial_m$ on $M$,
\begin{equation}
\label{eq:postscriptum-258-general}
    \pi_{\fh}
    \left(
    \frac{\partial T_m}{\partial t}
    \right)
    =
    \pi_{\fh}([T_m,C]).
\end{equation}
Taking $C=\Div T$ gives precisely \cite[Equation~(2.58)]{Fadel2022}, with $c_H=1$.

Let $I=[0,1]$ and consider the product manifold
\begin{equation*}
    \widetilde M=I\times M
\end{equation*}
with product metric $dt^2+g$. Then
\begin{equation*}
    T\widetilde M
    =
    \bR\partial_t\oplus TM.
\end{equation*}
We embed $\SO(n)$ into $\SO(n+1)$ by
\begin{equation*}
    \SO(n)\ni A
    \longmapsto
    \begin{pmatrix}
        1 & 0\\
        0 & A
    \end{pmatrix},
\end{equation*}
and define
\begin{equation*}
    \tilde H:=\{1\}\times H.
\end{equation*}
Accordingly,
\begin{equation*}
    \fso(n+1)
    =
    \fso(n)\oplus\mathfrak p,
    \qquad
    \mathfrak p\cong\bR^n.
\end{equation*}
If $\fh=\mathrm{Lie}(H)$ and $\fso(n)=\fh\oplus\fm$ is the reductive decomposition induced by $H$, then
\begin{equation*}
    \fso(n+1)
    =
    \tilde\fh\oplus\tilde\fm,
    \qquad
    \tilde\fh\cong\fh,
    \qquad
    \tilde\fm=\fm\oplus\mathfrak p.
\end{equation*}

The family $\xi_t$ defines a $\widetilde H$-structure $\widetilde\xi$ on $\widetilde M$ by adjoining the trivial direction $\mathbb R\partial_t$ to the evolving $H$-structure on the slices $\{t\}\times M$. Equivalently, at the level of model tensors, if $\xi_0$ is the model for $\xi$, then $\widetilde\xi_0$ is obtained from $\xi_0$ by adding the distinguished unit direction in $\mathbb R\oplus\mathbb R^n$. The stabiliser of $\widetilde\xi_0$ inside $\SO(n+1)$ is precisely $\widetilde H$. Indeed, for
\begin{equation*}
    \tilde h=
    \begin{pmatrix}
        1 & 0\\
        0 & h
    \end{pmatrix}
    \in\tilde H,
\end{equation*}
with $h\in H$, one has
\begin{equation*}
    \tilde h\cdot\tilde\xi_0
    =
    \tilde\xi_0,
\end{equation*}
because $h\cdot\xi_0=\xi_0$. Conversely, an element preserving the distinguished unit vector and the extended model tensor restricts to an element of $H$ on its orthogonal complement. Since the oriented orthonormal frame bundle of $(\widetilde M,dt^2+g)$ splits according to $T\widetilde M=\bR\partial_t\oplus TM$, this gives a genuine $\tilde H$-structure $\tilde\xi$ on $\widetilde M$.

Let $\tilde T$ denote the intrinsic torsion of $\tilde\xi$. We claim that
\begin{equation}
\label{eq:postscriptum-tilde-torsion}
    \tilde T_{a\partial_t+X}
    =
    aC+T_X^t.
\end{equation}
Indeed, using the product connection and the defining relation for intrinsic torsion, the $t$-component gives
\begin{equation*}
    \tilde\nabla_{\partial_t}\tilde\xi
    =
    \frac{\partial\xi_t}{\partial t}
    =
    C\diamond\xi_t
    =
    C\diamond\tilde\xi.
\end{equation*}
Similarly, for a vector field $X$ tangent to $M$,
\begin{equation*}
    \tilde\nabla_X\tilde\xi
    =
    \nabla_X\xi_t
    =
    T_X^t\diamond\xi_t
    =
    T_X^t\diamond\tilde\xi.
\end{equation*}
By linearity in the tangent vector, this proves \eqref{eq:postscriptum-tilde-torsion}. In particular,
\begin{equation*}
    \tilde T_t=C,
    \qquad
    \tilde T_m=T_m^t.
\end{equation*}

We now apply \cite[Equation~(1.56)]{Fadel2022} to the $\tilde H$-structure $\tilde\xi$ on $\widetilde M$, taking one index in the $\partial_t$-direction and one index $\partial_m$ tangent to $M$. Since the product metric has no mixed curvature terms, this gives
\begin{equation*}
    \pi_{\tilde\fh}
    \left(
    \tilde\nabla_t\tilde T_m
    -
    \tilde\nabla_m\tilde T_t
    \right)
    =
    -2
    \pi_{\tilde\fh}
    \left(
    [\tilde T_t,\tilde T_m]
    \right).
\end{equation*}
Using $\tilde T_t=C$, $\tilde T_m=T_m^t$, and $\tilde\fh=\fh$, we obtain
\begin{equation}
\label{eq:postscriptum-pre-258}
    \pi_{\fh}
    \left(
    \frac{\partial T_m}{\partial t}
    -
    \nabla_mC
    \right)
    =
    -2
    \pi_{\fh}
    \left(
    [C,T_m]
    \right).
\end{equation}

It remains to identify the $\fh$-part of $\nabla_mC$. Let $\nabla^H$ denote the canonical $H$-connection determined by the $H$-structure. With the convention of \cite{Fadel2022},
\begin{equation*}
    \nabla_m^H
    =
    \nabla_m+T_m.
\end{equation*}
Therefore, for an $\fso(TM)$-valued section $C$,
\begin{equation*}
    \nabla_m C
    =
    \nabla_m^H C-[T_m,C].
\end{equation*}
Since $C\in\Omega^2_{\fm}(M)$ and $\nabla^H$ preserves the splitting $\fso(TM)=\fh\oplus\fm$, we have
\begin{equation*}
    \nabla_m^H C\in\fm.
\end{equation*}
Hence
\begin{equation*}
    \pi_{\fh}(\nabla_mC)
    =
    -\pi_{\fh}([T_m,C]).
\end{equation*}
Substituting this into \eqref{eq:postscriptum-pre-258}, we get
\begin{align*}
    \pi_{\fh}
    \left(
    \frac{\partial T_m}{\partial t}
    \right)
    &=
    -2\pi_{\fh}([C,T_m])
    +
    \pi_{\fh}(\nabla_mC)
    \\
    &=
    -2\pi_{\fh}([C,T_m])
    -
    \pi_{\fh}([T_m,C])
    \\
    &=
    \pi_{\fh}([T_m,C]),
\end{align*}
because $[C,T_m]=-[T_m,C]$. This proves \eqref{eq:postscriptum-258-general}. Taking $C=\Div T$ gives the identity \cite[Equation~(2.58)]{Fadel2022} for arbitrary closed and connected $H\leqslant\SO(n)$.

\subsection{The harmonic flow for general $H$}\label{sec:postscriptum_flow_extensions}

Throughout this subsection, $(M^n,g)$ is a closed Riemannian manifold admitting a compatible $H$-structure. The harmonic-flow estimates in \cite[Part~2]{Fadel2022} use the structural condition \cite[Condition~(2.1)]{Fadel2022} and the evolution identity \cite[Equation~(2.58)]{Fadel2022}. The identity \eqref{eq:postscriptum-258-general} holds for general $H$. The estimates below use the intrinsic torsion energy and a uniform comparison of $|\nabla\xi|^2$ and $|T|^2$, without requiring Condition~(2.1).

The relevant norm comparison is the following. For any compatible $H$-structure, the intrinsic torsion satisfies
\begin{equation*}
    \nabla_X\xi
    =
    T_X\diamond\xi.
\end{equation*}
Pointwise, this is induced by the linear map
\begin{equation*}
    \fm\longrightarrow V,
    \qquad
    A\longmapsto A\diamond\xi_0.
\end{equation*}
Since $H$ is the stabiliser of the model tensor $\xi_0$, this map has trivial kernel on $\fm$. Being a linear injective map between finite-dimensional normed spaces, it gives constants $0<\tilde c\leqslant c$, depending only on the model $H$-structure, such that
\begin{equation}
\label{eq:postscriptum-nabla-torsion-comparison}
    \tilde c |T|^2
    \leqslant
    |\nabla\xi|^2
    \leqslant
    c|T|^2.
\end{equation}
Consequently, for
\begin{equation*}
    \mathcal E(\xi)
    =
    \frac12\int_M |T|^2\,\mathrm{vol}_g
\end{equation*}
and
\begin{equation*}
    \mathcal D(\xi)
    =
    \frac12\int_M |\nabla\xi|^2\,\mathrm{vol}_g,
\end{equation*}
one has
\begin{equation}
\label{eq:postscriptum-energy-comparison}
    \tilde c\,\mathcal E(\xi)
    \leqslant
    \mathcal D(\xi)
    \leqslant
    c\,\mathcal E(\xi).
\end{equation}

The second point is that $\mathcal E$ is the functional naturally adapted to arbitrary $H$. Indeed, by \cite[Corollary~1.46]{Fadel2022}, if $\frac{\partial\xi_t}{\partial t}=C\diamond\xi_t$ is an isometric variation, then
\begin{equation*}
    \left.
    \frac{d}{dt}
    \right|_{t=0}
    \mathcal E(\xi_t)
    =
    -
    \int_M
    \langle \Div T,C\rangle\,\mathrm{vol}_g.
\end{equation*}
Thus the harmonic flow
\begin{equation*}
    \frac{\partial\xi}{\partial t}
    =
    \Div T\diamond\xi
\end{equation*}
is the negative gradient flow of $\mathcal E$ with respect to the $L^2$ metric induced by the normal homogeneous metric on $\SO(n)/H$, and $\mathcal E$ is non-increasing along the flow. By contrast, the functional $\mathcal D$ is the corresponding gradient functional only under the additional hypothesis \cite[Condition~(2.1)]{Fadel2022}; without that hypothesis, $\mathcal D$ need not be monotone along the harmonic flow.

We next describe how these estimates enter the harmonic-flow arguments of \cite[Part~2]{Fadel2022}.

First, the short-time existence theorem \cite[Proposition~2.1]{Fadel2022} is unaffected, since it is a general result for the harmonic section flow. Likewise, the parabolic rescaling argument \cite[Lemma~2.4]{Fadel2022} only uses the harmonic flow equation and is independent of Condition~(2.1).

The Bochner-type and Shi-type estimates \cite[Lemma~2.2 and Proposition~2.3]{Fadel2022} are already available for $|T|^2$. The comparison \eqref{eq:postscriptum-nabla-torsion-comparison} then transfers the resulting estimates to $|\nabla\xi|^2$, with constants modified only by the model $H$-structure. In particular, the estimates may be used interchangeably for the two densities whenever only upper bounds and compactness consequences are needed.

The almost-monotonicity formula \cite[Theorem~2.5]{Fadel2022} and its integrated version \cite[Theorem~2.7]{Fadel2022} are naturally expressed in terms of $\mathcal E$ and $|T|^2$. Therefore their validity does not depend on Condition~(2.1). The $\epsilon$-regularity theorem \cite[Theorem~2.10]{Fadel2022} follows from the same almost-monotonicity argument. When the statement is formulated with $e(\xi)=|\nabla\xi|^2$, one obtains it from the corresponding $|T|^2$ estimate by using \eqref{eq:postscriptum-nabla-torsion-comparison}.

The energy gap result \cite[Proposition~2.11]{Fadel2022} may still be stated in its original form using $\mathcal D(\xi)$ and $|\nabla\xi|^2$. Indeed, the proof can be run with the $|T|^2$ versions of the Shi-type estimates and $\epsilon$-regularity, and the conclusion is then translated back to $\mathcal D$ and $|\nabla\xi|^2$ using \eqref{eq:postscriptum-nabla-torsion-comparison} and \eqref{eq:postscriptum-energy-comparison}.

In Section~2.5 of \cite{Fadel2022}, Lemma~2.20 works equally for $|T|^2$ and $|\nabla\xi|^2$. In the argument of \cite[Lemma~2.22]{Fadel2022}, we use the quantity
\begin{equation*}
    \bar e(t)
    =
    \max_M |T(\cdot,t)|^2.
\end{equation*}
The Shi-type estimates for $|T|^2$ give uniform bounds for $|\nabla^mT|^2$. Then, using the formulae relating $\nabla\xi$ and $T$ in the proof of \cite[Proposition~2.3]{Fadel2022}, one obtains uniform bounds for the derivatives of $\xi$. This gives smooth subsequential convergence of $\xi(t_j)$ along suitable sequences $t_j\to\infty$. The second part of the proof of Lemma~2.22 carries over with $\mathcal E(\xi)$ replacing $\mathcal D(\xi)$, using the monotonicity of $\mathcal E$ along the flow.

Since $|T|^2$ and $|\nabla\xi|^2$ have the same parabolic scaling, the hypotheses of \cite[Lemma~2.23]{Fadel2022} can be reformulated in terms of $|T|^2$ and $\mathcal E(\xi)$. The proofs of \cite[Theorems~2.13 and 2.16]{Fadel2022} then go through with $\mathcal E$ in place of $\mathcal D$, after changing the smallness constants according to \eqref{eq:postscriptum-energy-comparison}. Thus the long-time existence and subconvergence theorem, as well as the finite-time singularity criterion in a homotopy class containing no torsion-free structure, remain valid in this broader setting.

It remains to consider the uniqueness of the long-time limit, treated in Section~2.6 of \cite{Fadel2022}. There, the additional hypothesis was precisely Equation~(2.58). Since this identity is now known to hold for arbitrary closed and connected $H\leqslant\SO(n)$ by \eqref{eq:postscriptum-258-general}, the proof of \cite[Proposition~2.24]{Fadel2022} applies in general, after replacing $\mathcal D$ by $\mathcal E$ where monotonicity is used. In particular, the convergence estimate remains valid, and the harmonic flow has a unique long-time limit under the corresponding smallness assumptions.

The same replacement applies to \cite[Theorem~2.27]{Fadel2022}. Its proof uses the monotonicity of the energy, the small-torsion estimates, and the uniqueness mechanism from Proposition~2.24. Working with $\mathcal E$ and $|T|^2$, and then translating the conclusion via \eqref{eq:postscriptum-nabla-torsion-comparison}, gives the same result for arbitrary closed and connected $H\leqslant\SO(n)$.

Finally, the proof of \cite[Theorem~2.29]{Fadel2022} also extends. Part~(i) is reproduced using $|T|^2$ and the comparison estimate \eqref{eq:postscriptum-nabla-torsion-comparison} in place of Condition~(2.1). For part~(ii), one still measures closeness in terms of $\nabla\xi$, but the estimates are obtained from $\mathcal E(\xi)$, which decreases along the flow, together with \eqref{eq:postscriptum-nabla-torsion-comparison}. Equivalently, the monotone quantity is now the $L^2$-norm of the intrinsic torsion, and the corresponding control of $|\nabla\xi_t|^2$ follows from the uniform comparison. This yields the required $C^0$ bounds for $|\nabla\xi_t|^2$ and allows the argument of \cite{Fadel2022} to conclude as before.

For arbitrary closed and connected $H\leqslant\SO(n)$, the energies $\mathcal D$ and $\mathcal E$ satisfy the uniform comparison \eqref{eq:postscriptum-energy-comparison}, and $\mathcal E$ is monotone along the harmonic flow. Together with the evolution identity \eqref{eq:postscriptum-258-general}, these properties give the stated extensions of the estimates in \cite[Part~2]{Fadel2022}.

\bmhead{Acknowledgements}
The first author would like to thank Jakob Stein for discussions concerning the connectedness of the space of $H$-structures, which motivated the results of Section~\ref{sec: topology_of_H_structures}. The second author thanks Andr\'es J. Moreno for discussions concerning the evolution identity for isometric flows in Section~\ref{subsec:evolution_identity_isometric_flows}.

Daniel Fadel was partially supported by the University of S\~ao Paulo New Faculty Support Program. The authors also acknowledge support from the MATH-AmSud project ``Symmetries in Geometry and Physics'' (SGP 24-MATH-12) and from the Brazil--France PRCI ANR-FAPESP project ``BRIDGES -- Brazil-France interplays in Gauge Theory, extremal structures and stability'' (ANR-21-CE40-0017).

\emph{AI declaration.} ChatGPT was used to assist with LaTeX editing, language revision, and preliminary consistency checks during the preparation of this manuscript. The authors have reviewed the mathematical content and take full responsibility for its correctness.

\bibliography{sn-bibliography}

\end{document}